\documentclass[12pt]{amsart}
\usepackage{amsmath,amsthm,amssymb}
\usepackage{xcolor}
\usepackage[colorlinks]{hyperref}

\newtheorem{theorem}{Theorem}[section]
\newtheorem{lemma}[theorem]{Lemma}
\newtheorem{corollary}[theorem]{Corollary}
\newtheorem{remark}[theorem]{Remark}

\theoremstyle{definition}
\newtheorem{definition}[theorem]{Definition}

\numberwithin{equation}{section}

\newcommand{\E}{{\mathbb{E}}}
\newcommand{\D}{{\mathbb{D}}}

\begin{document}

\title[Improved bounds for the total variation distance]
{Improved bounds for the total variation distance between stochastic polynomials}

\author{Egor Kosov}

\address{\noindent Egor Kosov:
Steklov Mathematical Institute of Russian Academy of Sciences, Moscow 119991, Russia
and
National Research University Higher School of Economics, Moscow, Russia \newline
Email: ked\_2006@mail.ru}

\author{Anastasia Zhukova}

\address{
Anastasia Zhukova:
Faculty of Mechanics and Mathematics, Lomonosov Moscow State University, Moscow, 119991 Russia}

\maketitle

\begin{abstract}
The paper studies upper bounds for the total variation distance
between two polynomials of a special form in random vectors satisfying the Doeblin-type condition.
Our approach is based on the recent results concerning Nikolskii--Besov-type smoothness of
distribution densities of polynomials in logarithmically concave random vectors.
The main results of the paper improve previously obtained estimates
of Nourdin--Poly and Bally--Caramellino.
\end{abstract}



\section{Introduction}\label{sect-1}

Recall that the total variation distance $d_{\rm TV}(\xi, \eta)$
between random variables $\xi, \eta$ is defined by the equality
$$
d_{\rm TV}(\xi, \eta) :=
\sup\Bigl\{\mathbb{E}\bigl[\varphi(\xi) - \varphi(\eta)\bigr]
\colon \varphi\in C_0^\infty(\mathbb{R}),
\ \|\varphi\|_\infty \le1 \Bigr\}.
$$
In this paper we study upper bounds for the total variation distance
between the distributions of two polynomials of the form
\begin{align}\label{form}
&Q_{d, k_*}(a, X)\\
&=\sum_{m=0}^d  \sum_{n_1 < \ldots < n_m}   \sum_{\substack{k_1, \ldots, k_m = 1}}^{k_*}
\sum_{j_1, \ldots, j_m = 1}^{N}\!\!a_m((n_1, k_1, j_1), \ldots , (n_m, k_m, j_m))
\prod_{i=1}^m (X_{n_i, j_i} ^{k_i}\! - \E \bigl[X_{n_i, j_i} ^{k_i}\bigr]) \nonumber
\end{align}
where $a=(a_0, \ldots, a_d)$, $a_m\colon (\mathbb{N}^3)^m\to \mathbb{R}$,
and where
$X:=\{X_n\}_{n=1}^\infty$, $X_{n}:=(X_{n,1}, \ldots, X_{n, N})$,
is a sequence of $N$-dimensional independent
random vectors
satisfying the so called Doeblin’s condition, i.e.
the distribution of each $X_k$ is lower bounded by the Lebesgue measure
on some ball in $\mathbb{R}^N$ (see Definition \ref{D-Doebl} below).
Recently, upper bounds for different distances between distributions of polynomials
have been extensively studied
(see e.g. \cite{BC17}, \cite{BCP}, \cite{NPR10}, \cite{NNP}, \cite{NPec09},
\cite{NP13}, the book \cite{NPec12}, and
survey \cite{B22}).

The present paper continues the research started in \cite{MOO}, \cite{NP15} and \cite{BC19}.
The approach to obtaining upper bounds for the total variation distance
developed in the last two cited papers consists of two main technical steps.
Firstly, one notices that the distributions of the random vectors $X_j$, satisfying Doeblin’s condition,
coincide with the distributions of random vectors of the form
\begin{equation}\label{rep}
\varepsilon_j V_j + (1 - \varepsilon_j)U_j
\end{equation}
where $\varepsilon_j, V_j, U_j$ are independent and
$\varepsilon_j$ is a Bernoulli random variable,
$V_j$ is a random vector with sufficiently good density,
$U_j$ is some other random vector.
Next, working conditionally with respect to
$(\varepsilon, U)=(\varepsilon_1,\ldots, \varepsilon_n, U_1, \ldots, U_n)$,
one develops a Malliavin-type calculus with respect to the distribution of
the random vector
$V=(V_1, \ldots, V_n)$. This Malliavin-type calculus
is then used to obtain the desired upper bounds for the total variation distance.

The present paper has two main goals. The first one is to simplify the described technique
involved in the study of the total variation distance upper bounds.
In particular, instead of the Malliavin-type calculus,
we use some recent results concerning the smoothness properties of distributions of polynomials in
logarithmically concave random vectors
from \cite{Kos} (see also \cite{BKZ} and
\cite{Kos-Adv}).
The second goal of this paper
is
to provide sharper results compared to the main
theorems from
\cite{NP15} and \cite{BC19}.

Let $a=(a_0, \ldots, a_d)$ be a collection of coefficient functions
$a_m\colon (\mathbb{N}^3)^m\to \mathbb{R}$. Everywhere further through out the paper we always
assume that only a finitely many coefficients $a_m(n,k,j)$ are non zero.
Let
$$
[a_m]:=\Bigl( \sum_{n_1 < \ldots < n_m}   \sum_{k_1, \ldots, k_m = 1}^{k_*}
\sum_{j_1, \ldots, j_m = 1}^{N}
\bigl[a_m((n_1, k_1, j_1), \ldots , (n_m, k_m, j_m))\bigr]^2\Bigr)^{\frac{1}{2}},
$$
$$
[a_{l, d}]:=
\Bigl(\sum_{m=l}^d[a_m]^2\Bigr)^{\frac{1}{2}},\quad\quad\quad
[a]:=[a_{0, d}],
$$
$$
\delta(a_m): =\max_n
 \sum_{n_1 < \ldots < n_m}   \sum_{k_1, \ldots, k_m = 1}^{k_*}
\sum_{j_1, \ldots, j_m = 1}^{N}
I_{\{n\in\{n_1, \ldots, n_m\}\}}
\bigl[a_m((n_1, k_1, j_1), \ldots , (n_m, k_m, j_m))\bigr]^2,
$$
$$
\delta(a): = \max_n
\sum_{m=0}^d  \sum_{n_1 < \ldots < n_m}   \sum_{k_1, \ldots, k_m = 1}^{k_*}
\sum_{j_1, \ldots, j_m = 1}^{N}\!\!
I_{\{n\in\{n_1, \ldots, n_m\}\}}
\bigl[a_m((n_1, k_1, j_1), \ldots , (n_m, k_m, j_m))\bigr]^2.
$$
Our benchmark is the following
theorem from \cite{BC19}
(in the formulation of the theorem, $d_1$ is the Kantorovich--Rubinstein distance and $d_k$ is
its generalization, introduced below in Section \ref{sect-2}).
\begin{theorem}[see Theorem 3.3 in \cite{BC19}]\label{T-BC}
Let $A, B\in (0, +\infty)$,
and $k, d, d', l, l', N, k_*\in \mathbb{N}$, $l'<d', l<d$.
Assume that each element in sequences of independent random vectors
$\{X_n\}_{n=1}^\infty$ and $\{Y_n\}_{n=1}^\infty$
satisfy Doeblin-type condition and
$$
\sup\limits_{n}\max\limits_{1\le j\le N}\mathbb{E}\bigl[|X_{n, j}|^p\bigr]<\infty\quad  \sup\limits_{n}\max\limits_{1\le j\le N}\mathbb{E}\bigl[|Y_{n, j}|^p\bigr]<\infty\quad \forall p>1.
$$
Assume that two collections of
coefficient functions $a=(a_0, \ldots, a_d)$, $b=(b_0, \ldots, b_{d'})$ satisfy the following
assumptions:
$$
A\le [a_l], [b_{l'}];\quad [a], [b]\le B
$$
and $[a_{l+1, d}]$, $[b_{l'+1, d'}]$ are sufficiently small.
Then for each $\theta\in((1+k)^{-2},1)$
there are numbers $C, c > 0$
such that
\begin{multline}\label{BC-bound}
d_{\rm TV} \bigl(Q_{d, k_*}(a, X), Q_{d', k_*}(b, Y)\bigr)
\le C\Bigl[\bigl[d_k\bigl(Q_{d, k_*}(a, X), Q_{d', k_*}(b, Y)\bigr)\bigr]^{\frac{\theta}{1+ 2k\max\{l, l'\} k_*}}
+\\+ \exp\bigl(-c[\delta (a)]^{-1}\bigr) +
\exp\bigl(-c[\delta (b)]^{-1}\bigr)
+[a_{ l+1, d}]^{\frac{2\theta}{\max\{l, l'\}k_*}} +
[a_{l'+1, d'}]^{\frac{2\theta}{\max\{l, l'\}k_*}}\Bigr].
\end{multline}
\end{theorem}

The approach, employed in the present paper, allows us to obtain sharper estimates in Theorem \ref{T-BC}.
Moreover, this approach allows to simplify the most technical parts
of the proofs from \cite{NP15} and \cite{BC19} that are based on the Malliavin-type calculus.
One of our main theorems asserts the following.

\begin{theorem}[see Theorem \ref{T-2} below]
Let $A, B\in (0, +\infty)$,
and $k, d, d', N, k_*\in \mathbb{N}, d_* = \max (d, d'), \omega>0, r>0, R>0$.
Then there are numbers
$C:=C(A, B, k, d,d', N, k_*, \omega, r, R)$ and $c:=c(A, B, d,d', N, k_*, \omega, r)$, dependent
only on the listed parameters,
such that, for every pair of sequences
$\{X_n\}_{n=1}^\infty$ and $\{Y_n\}_{n=1}^\infty$
of independent $N$-dimensional random vectors satisfying Doeblin-type condition, and for every pair of
coefficient function collections
$a=(a_0, \ldots, a_d)$, $a_m\colon (\mathbb{N}^3)^m\to \mathbb{R}$,
$b=(b_0, \ldots, b_{d'})$, $b_m\colon (\mathbb{N}^3)^m\to \mathbb{R}$
such that
$$
A\le [a_d], [b_{d'}];\quad [a_d], [b_{d'}]\le B
$$
one has
\begin{multline*}
d_{\rm TV} \bigl(Q_{d, k_*}(a, X), Q_{d', k_*}(b, Y)\bigr)
\le C\bigl[d_{k}\bigl(Q_{d, k_*}(a, X), Q_{d', k_*}(b, Y)\bigr)\bigr] ^{\frac{1}{1+d_* kk_*}}
\\
+ 12d \exp\bigl(-c[\delta (a_d)]^{-1}\bigr) +
12d' \exp \bigl(-c[\delta (b_{d'})]^{-1}\bigr).
\end{multline*}
\end{theorem}

We point out, that in this theorem the power of the $d_k$ distance is almost twice bigger
than in Theorem \ref{T-BC}. Apart from that, the bound does not involve moments of random vectors $X_n, Y_n$.

The paper is organized as follows.
Section \ref{sect-2} contains all the necessary definitions, notation, and
several known preliminary theorems. In Section \ref{sect-3}
we consider the simplest case when $N = k_*=1$,
which corresponds to the results of \cite{NP15}.
In Section \ref{sect-4} we provide
bounds for the total variation distance
between general stochastic polynomials that
improve the estimate \eqref{BC-bound} from Theorem~\ref{T-BC}.
In Section \ref{sect-5} we provide bounds for the
Fourier transform of distributions of polynomials
of the form
\eqref{form}.

\section{Definitions, notation and preliminary facts}\label{sect-2}

Let $\xi$ and $\eta$ be two random variables.
The Kantorovich--Rubinstein distance $d_{\rm KR}(\xi, \eta)$
between
the distributions of random variables $\xi$ and $\eta$ is defined by the following equality:
$$
d_{\rm KR}(\xi, \eta):= \sup\biggl\{\mathbb{E}\bigl[\varphi(\xi)-\varphi(\eta)\bigr]\colon
\varphi\in C_0^\infty(\mathbb{R}),\ \|\varphi\|_\infty \le 1,\ \|\varphi'\|_\infty\le1\biggr\}.
$$
It is known (see, e.g \cite{B18}) that the convergence of distributions in the
Kantorovich--Rubinstein distance is equivalent to the convergence in distribution of random variables
 (i.e. the week convergence of these distributions).
We will also consider the following generalization of the Kantorovich--Rubinstein distance:
$$
d_k(\xi, \eta):=\sup\biggl\{\mathbb{E}\bigl[\varphi(\xi)-\varphi(\eta)\bigr]\colon
\varphi\in C_0^\infty(\mathbb{R}),\ \|\varphi\|_\infty \le 1,\ \|\varphi'\|_\infty\le1,\ldots,
\|\varphi^{(k)}\|_\infty\le1\biggr\}.
$$

Throughout the paper we use the standard notation for the expectation $\mathbb{E}\xi$
and for the variance $\mathbb{D}\xi := \mathbb{E}(\xi - \mathbb{E}\xi)^2$
of a random variable $\xi$.
For a function  $\varphi\in C_0^\infty(\mathbb{R})$,
we denote
$\|\varphi\|_\infty:=\max\limits_{t\in \mathbb{R}}|\varphi(t)|$.

We recall that a random vector $X$ is called logarithmically concave
if
$$
P\bigl(X\in (tA+(1-t)B)\bigr)\ge \bigl[P(X\in A)\bigr]^t\bigl[P(X\in B)\bigr]^{1-t}
\quad \forall t\in(0, 1)
$$
for every pair of Borel sets $A$ and $B$.
We point out that linear images of logarithmically concave random vectors
are also logarithmically concave random vectors.
Moreover, a random vector with independent logarithmically concave components
is a logarithmically concave random vector.
If a random vector $X$ has a density $\varrho$ with respect to the Lebesgue measure,
then it is logarithmically concave if and only if $-\ln\varrho$ is a convex function
(see \cite{BGVV} for more details concerning
logarithmically concave measures).

We now introduce the main assumption
(the so called Doeblin-type condition)
on random vectors studied in this paper.

\begin{definition}\label{D-Doebl}
{\rm
Let $N\in\mathbb{N}, \omega > 0, r>0, R>0$.
We say that the random vector $X$ in $\mathbb{R}^N$
satisfy the condition $\mathfrak{D}(\omega, r, R)$
if there is a point $x\in \mathbb{R}^N$ such that $|x|\le R$ and
for every Borel set $A \subset B_r(x)$ one has the estimate
$$
P (X \in A) \ge \omega \lambda (A)
$$
where $\lambda$ is the Lebesgue measure on $\mathbb{R}^N$ and where $B_r(x)$
is the Euclidean ball of radius $r$ centered at the point $x$.
}
\end{definition}

This Doeblin-type condition is a classical assumption to study convergence
in total variation of distributions. Similar assumptions in a weaker form appear
in the seminal paper of Yu. V. Prohorov \cite{P52}
where he studied the convergence in the CLT in the total variation distance.

\vskip .1in

We now formulate several already known results
that will be used throughout the paper.

\begin{theorem}[see \cite{Kos}]\label{t2.1}
Let $d\in \mathbb{N}$.
There is a number $C(d)>0$, dependent only on $d$, such that
for every logarithmically concave random vector
$V=(V_1, \ldots, V_N)$ in $\mathbb{R}^N$ and for every non-constant
polynomial
$$
f(V):= \sum_{j_1+\cdots+j_N\le d}a(j_1,\ldots, j_N)V_1^{j_1}\cdots V_N^{j_N}
$$
of a degree not greater than $d$,
and for every function
$\varphi \in C_b^{\infty}(\mathbb{R})$,
$\| \varphi \| _{\infty} \le 1$,
one has
$$
\mathbb{E}[\varphi ' (f(V))]
\le \frac{C(d)}{(\mathbb{D}[f(V)])^{\frac{1}{2d}}}\cdot \| \varphi' \|^{1 - \frac{1}{d}} _{\infty}
$$
and (equivalently)
$$
d_{\rm TV}\bigl(f(V) + h, f(V)\bigr)
\le \frac{C(d)}{(\mathbb{D}[f(V)])^{\frac{1}{2d}}}\cdot |h|^{\frac{1}{d}}\quad \forall h\in \mathbb{R}.
$$
\end{theorem}

\begin{remark}\label{rm2.1}
{\rm
The result of Theorem \ref{t2.1}
actually means that the density
$\varrho_f$ of the random variable
$f(V)$ belongs to the Nikolskii--Besov space~$B_{1,\infty}^{1/d}(\mathbb{R})$
(see also \cite{Kos-FCAA}, \cite{Kos-MS} for a more detailed discussion of this phenomenon).
}
\end{remark}

We will also need the following generalization of
the previous theorem from \cite{Kos-IMRN}.

\begin{theorem}[see \cite{Kos-IMRN}]\label{t2.1.1}
Let $d\in \mathbb{N}$.
There is a number $C(d)>0$, dependent only on $d$, such that
for every logarithmically concave random vector
$V=(V_1, \ldots, V_N)$ in $\mathbb{R}^N$
and for every pair of
polynomials
$$
f(V):= \sum_{j_1+\cdots+j_N\le d}a(j_1,\ldots, j_N)V_1^{j_1}\cdots V_N^{j_N},\quad
g(V):= \sum_{j_1+\cdots+j_N\le d}b(j_1,\ldots, j_N)V_1^{j_1}\cdots V_N^{j_N}
$$
of degrees not greater than $d$, one has the bound
$$
(\mathbb{D}[f(V)])^{\frac{1}{2d}}\cdot
d_{\rm TV}\bigl(f(V), g(V)\bigr)
\le C(d)\cdot \bigl(\mathbb{E}|f(V) - g(V)|^2\bigr)^{1/2d}.
$$
\end{theorem}

\begin{remark}\label{rm2.2}
{\rm
We note that from the proofs in \cite{Kos-IMRN}
actually follows that in the previous theorem,
one has the bound
$$
(\mathbb{D}[f(V)])^{\frac{1}{2d}}\cdot
d_{\rm TV}\bigl(f(V), g(V)\bigr)
\le C(d, d')\cdot
\bigl(\mathbb{E}|f(V) - g(V)|^2\bigr)^{1/2d}
$$
when
$$
f(V):= \sum_{j_1+\cdots+j_N\le d}a(j_1,\ldots, j_N)V_1^{j_1}\cdots V_N^{j_N},\quad
g(V):= \sum_{j_1+\cdots+j_N\le d'}b(j_1,\ldots, j_N)V_1^{j_1}\cdots V_N^{j_N},
$$
i.e. when $f$ is of a degree not greater than $d$ and
$g$ is of a degree not greater than $d'$.
}
\end{remark}

\begin{theorem}[see Lemma A.1 in \cite{BC19}]\label{t2.2}
Let
$$
\Phi(x) =
\sum_{1\le j_1<\ldots<j_d}c(j_1,\ldots, j_d)^2x_{j_1}\cdot\ldots\cdot x_{j_d},
$$
let
$$
[c]^2:=\sum\limits_{1\le j_1<\ldots<j_d}c(j_1,\ldots, j_d)^2,\quad
\delta(c):= \max_j
\sum\limits_{\substack{1\le j_1<\ldots<j_d\\j\in\{j_1, \ldots, j_d\}}}
c(j_1,\ldots, j_d)^2,
$$
and let $\varepsilon_n$, $n \in \mathbb{N}$, be a sequence of
independent Bernoulli random variables with values $\{0, 1\}$ such that
$P(\varepsilon_n=1)=p\in (0, 1)$.
Then for every
$$
\theta\in\bigl(0, \bigl(\tfrac{p}{2}\bigr)^d[c]^2\bigr),
$$
one has the estimate
$$
P\bigl(\Phi(\varepsilon)\le \theta\bigr)\le 6d\exp\Bigl(-\frac{\theta^2}
{\delta(c)[c]^2}\Bigr).
$$
\end{theorem}

We will make use of the so-called Nummelin’s splitting (see \cite{Num78})
already applied in \cite{BC19}
(see also \cite{BC16} and \cite{BCP18} where this splitting was also used).

\begin{theorem}[see the discussion before Lemma 3.1 and Lemma 3.2 in \cite{BC19}]\label{t2.3}
Let $\omega\!>\!0$, $r\!>\!0$, $R\!>\!0$,
$N, k_*\in\mathbb{N}$
and assume that the random vector $X_0$ in $\mathbb{R}^N$
satisfies the condition $\mathfrak{D}(\omega, r, R)$.
Then
\begin{equation}\label{rep2}
X_0 \stackrel{Law}{=} \varepsilon_0 V_0 + (1 - \varepsilon_0)U_0
\end{equation}
where $\varepsilon_0$, $V_0$, $U_0$ are independent,
$\varepsilon_0$ is a Bernoulli random variable
with the probability of success $p = p(\omega, r, N)$,
$V_0$ is a logarithmically concave random vector,
and $U_0$ is another random vector.
Moreover, there is a number $\alpha = \alpha(r, R, N, k_*)$, dependent only on
the parameters $r$, $R$, $N$ and $k_*$, such that,
for a sequence $\{X_n\}_{n=1}^\infty$  of independent
random vectors in $\mathbb{R}^N$,
each of which
satisfies the condition $\mathfrak{D}(\omega, r, R)$,
and for a collection of coefficients
$a=(a_1, \ldots, a_d)$, $a_m\colon (\mathbb{N}^3)^m\to \mathbb{R}$,
one has
$$
\E [|Q_{d, k_*}(a, V)|^2] \ge \alpha^d [a]^2
$$
where $V=(V_1, \ldots, V_n, \ldots)$
and $V_j$ is the random vector
from the representation \eqref{rep2}
for the random vector $X_j$.
\end{theorem}

\begin{theorem} [see Theorem 2.2 in \cite{BC19}]\label{t2.4}
Let $ \{ Z_n \}_{n=1} ^ \infty $ be a sequence of centered independent random vectors in
$\mathbb{R}^N$.
Let
\begin{equation}\label{eqM}
M_q(Z):= \max\bigl\{1, \sup\{[\mathbb{E} |Z_{n, i}|^q]^{1/q}\colon
i \in \{1, \ldots, N \}, n\in\mathbb{N}\bigr\}.
\end{equation}
Let $ \{ G_n \}_{n=1} ^ \infty $ be a sequence of centered independent Gaussian random vectors in
$ \mathbb{R} ^N$ such that
$\mathbb{E} [G_{n,i}G_{n,j}] = \mathbb{E} [Z_{n,i}Z_{n,j}]$. Then
there is a number
$C:=C\bigl(d, N, M_3(Z), M_3(G)\bigr)$,
dependent only on the listed parameters,
such that for every random variables
$$
S(c,Z) = \sum_{m=0}^d  \sum _{n_1 < \dots < n_m}
\sum_{j_1 = 1, \dots, j_m = 1}^N c_m((n_1, j_1), \ldots , (n_m, j_m))
\prod _{i=1} ^m Z_{n_i, j_i},$$
$$
S(c,G) = \sum_{m=0}^d  \sum _{n_1 < \dots < n_m}
\sum_{j_1 = 1, \dots, j_m = 1}^N c_m((n_1, j_1), \ldots , (n_m, j_m))
\prod _{i=1} ^m G_{n_i, j_i},
$$
one has
$$
d_3\bigl(S (c,Z), S (c,G)\bigr) \le
C [c]^2\sqrt{\delta (c).}$$
\end{theorem}

\section{The case $N=k_*=1$}\label{sect-3}

In this section we consider the simplest situation
of multilinear polynomials.
To present the main result of this section
more explicitly, for $p\in(0, 1)$ and $\alpha>0$,
we consider the class
$\mathcal{C}(p, \alpha)$ of all random
variables $X$ such that
$$
X \stackrel{Law}{=}
\varepsilon(\alpha V+x_0) + (1-\varepsilon)U
$$
where $x_0\in \mathbb{R}$, and
where $\varepsilon$, $V$, $U$ are independent,
$V$
is uniformly distributed on $[-1, 1]$,
$\varepsilon$ is a Bernoulli random variable
with values
$\{0, 1\}$ and $P(\varepsilon=1)=p$,
$U$ is some other random variable.
It was noted in \cite{NP15} (see Proposition 1.6)
that for every random variable
$X\in \mathfrak{D}(\omega, r, R)$,
one also has
$X\in \mathcal{C}(2\omega r, r)$.
The following theorem is the main result
of this section.

\begin{theorem}\label{T-1}
Let $k, d \in \mathbb{N}$, $\alpha>0$, $p\in(0, 1)$, $\varkappa>0$.
Then there is a positive number $C:=C(k, d, \alpha, p, \varkappa)$ such that, for every $n\in \mathbb{N}$, for every pair of random vectors
$X=(X_1, \ldots, X_n)$, $Y=(Y_1, \ldots, Y_n)$
with independent components $X_j, Y_j\in \mathcal{C}(p, \alpha)$ and for
every pair of polynomials
$$
Q_{d}(a, X) =\!\!\! \sum_{1\le n_1< \ldots < n_d \le n}
\!\!\!a (n_1, \dots , n_d) X_{n_1} \ldots X_{n_d},
\ Q_{d}(b, Y) =\!\!\!
\sum _{1\le n_1< \ldots < n_d \le n}\!\!\! b(n_1, \dots , n_d) Y_{n_1} \ldots Y_{n_d},
$$
such that
$[a] \ge \varkappa$ and
$[b]\ge \varkappa$,
one has the estimate
\begin{multline*}
d_{\rm TV}\bigl(Q_{d}(a, X), Q_{d}(b, Y)\bigr) \\ \le
C\bigl[d_k\bigl(Q_{d}(a, X), Q_{d}(b, Y)\bigr)\bigr]^{\frac{1}{1+kd}}
+ 12d\exp\bigl(-\tfrac{p^{2d}[a]^2}{16^d\delta(a)}\bigr)
+ 12d\exp\bigl(-\tfrac{p^{2d}[b]^2}{16^d\delta(b)}\bigr)
\end{multline*}
where the constant $C>0$ is of the form
$$
C:=C(k, d, \alpha, p, \varkappa)=
4C_k + \tfrac{C(d)}{\alpha \sqrt{p}}\varkappa^{-\frac{1}{d}},\
C_k = \max\{c_0, c_1, \ldots, c_k\},\
c_k =
\tfrac{1}{\sqrt{2\pi}}\int_{\mathbb{R}}\Bigl|\frac{d^k}{ds^k}e^{-\frac{s^2}{2}}\Bigr|\, ds.
$$
\end{theorem}

\begin{proof}
Let $\eta \in (0,1), Z \sim \mathcal{N}(0,1)$ and
assume that $Z$ is independent from $X$ and $Y$.
Then
\begin{multline*}
d_{\rm TV}\bigl(Q_{d}(a, X), Q_{d}(b, Y)\bigr)
\le
d_{\rm TV}\bigl(Q_{d}(a, X), Q_{d}(a, X)+\eta Z\bigr)
\\
+ d_{\rm TV}\bigl(Q_{d}(a, X)+\eta Z, Q_{d}(b, Y)+\eta Z\bigr)
+ d_{\rm TV}\bigl(Q_{d}(b, Y), Q_{d}(b, Y)+\eta Z\bigr)
\end{multline*}
and we now consider all the terms separately.

Let $\varphi\in C_0^\infty(\mathbb{R})$, $\|\varphi\|_\infty\le 1$.
Then
$$
\mathbb{E}\bigl[\varphi(Q_{d}(a, X)+\eta Z) - \varphi(Q_{d}(b, Y)+\eta Z)\bigr]
= \mathbb{E}\bigl[\varphi_\eta(Q_{ n}(a, X)) - \varphi_\eta(Q_{d}(b, Y))\bigr]
$$
where
$$
\varphi_\eta(t) = \int_{\mathbb{R}}\varphi(t+\eta s)
\tfrac{1}{\sqrt{2\pi}}e^{-\frac{s^2}{2}}\, ds.
$$
We note that
\begin{multline*}
\varphi_\eta^{(k)}(t) = \int_{\mathbb{R}}\varphi^{(k)}(t+\eta s)
\tfrac{1}{\sqrt{2\pi}}e^{-\frac{s^2}{2}}\, ds =
\eta^{-k}\int_{\mathbb{R}}\frac{d^k}{ds^k}\varphi(t+\eta s)
\tfrac{1}{\sqrt{2\pi}}e^{-\frac{s^2}{2}}\, ds
\\=
\eta^{-k} (-1)^k
\int_{\mathbb{R}}\varphi(t+\eta s)
\tfrac{1}{\sqrt{2\pi}}\frac{d^k}{ds^k}e^{-\frac{s^2}{2}}\, ds.
\end{multline*}
Thus,
$$
|\varphi_\eta^{(k)}(t)|\le c_k\eta^{-k},\quad
c_k =
\tfrac{1}{\sqrt{2\pi}}\int_{\mathbb{R}}\Bigl|\frac{d^k}{ds^k}e^{-\frac{s^2}{2}}\Bigr|\, ds.
$$
Since $\eta\in(0, 1)$, then
$$
d_{\rm TV}\bigl(Q_{ d}(a, X)+\eta Z, Q_{ d}(b, Y)+\eta Z\bigr)
\le C_k\eta^{-k}d_k\bigl(Q_{ d}(a, X), Q_{ d}(b, Y)\bigr)
$$
where $C_k = \max\{c_0, c_1, \ldots, c_k\}$.

We now estimate $d_{\rm TV}\bigl(Q_{ d}(a, X), Q_{ d}(a, X)+\eta Z\bigr)$.
We again consider $\varphi\in C_0^\infty(\mathbb{R})$, $\|\varphi\|_\infty\le 1$.
We recall that $X=(X_1, \ldots, X_n)$,
$X_j \stackrel{Law}{=}  \varepsilon_j(\alpha V_j+x_{0, j}) + (1-\varepsilon_j)U_j$,
and all the random variables $\varepsilon_j, U_j, V_j$ are mutually independent and
independent for all $j$.
We note that, for fixed
$x_0=(x_{0, 1}, \ldots, x_{0, n})$,
$\varepsilon=(\varepsilon_1,\ldots, \varepsilon_n)$ and $U=(U_1,\ldots, U_n)$,
the random variable $Q_{d}(a, X)$
is a polynomial of a degree not greater than $d$
in random variables $V = (V_1, \ldots, V_n)$.
We also point out that the vector $V$
has a logarithmically concave distribution.
Then, by Theorem \ref{t2.1},
for every $\theta>0$, there is an estimate
\begin{multline*}
\mathbb{E}\bigl[\varphi(Q_{d}(a, X)) - \varphi(Q_{d}(a, X)+\eta Z)\bigr]
\\=
\mathbb{E}_{\varepsilon, U, Z}\Bigl[
\bigl(I_{\{ \mathbb{D}_V[Q_{d}(a, X)] \le \theta \}}+ I_{\{ \mathbb{D}_V[Q_{d}(a, X)] > \theta\}}\bigr)
\mathbb{E}_V\bigl[\varphi(Q_{d}(a, X)) - \varphi(Q_{d}(a, X)+\eta Z))\bigr]\Bigr]
\\
\le 2P_{\varepsilon, U}\bigl(\mathbb{D}_V[Q_{d}(a, X)] \le \theta\bigr)+
C(d)\mathbb{E}_{\varepsilon, U, Z}
\Bigl[|\eta Z|^\frac{1}{d}\bigl(\D_V[Q_{ n}(a, X)]\bigr)^{-\frac{1}{2d}}I_{\{ \mathbb{D}_V[Q_{d}(a, X)] > \theta\}}\Bigr]
 \\ \le
2P_{\varepsilon, U}\bigl(\mathbb{D}_V[Q_{d}(a, X)] \le \theta\bigr) +
C(d)\eta^\frac{1}{d}\theta^{-\frac{1}{2d}}.
\end{multline*}
It remains to estimate the probability $P_{\varepsilon, V}\bigl(\mathbb{D}_U[Q_{ n}(a, X)] \le \theta\bigr)$.
We note that
(see the equalities $(3.12)$ and
the estimates after them
in \cite{NP15})
$$
\mathbb{D}_V[Q_{d}(a, X)] \ge
\alpha^{2d} 3^{-d}(d!)^2 \sum _{1\le j_1<\ldots<j_d \le n}
a^2 (j_1, \dots , j_d)
\varepsilon_{j_1} \ldots \varepsilon_{j_d}.
$$
Thus,
$$
P_{\varepsilon, U}\bigl(\mathbb{D}_V[Q_{ d}(a, X)] \le \theta\bigr)
\le
P\Bigl(\sum _{1\le j_1<\ldots<j_d \le n}
a^2 (j_1, \dots , j_d) \varepsilon _{j_1} \ldots \varepsilon _{j_d}
\le 3^d\alpha^{-2d}(d!)^{-2}\theta\Bigr).
$$
We take
$$
\theta = (\tfrac{\alpha^2 p}{12})^d(d!)^2[a]^2\ge (\tfrac{\alpha^2 p}{12})^d(d!)^2\varkappa^2.
$$
Then $3^d\alpha^{-2d}(d!)^{-2}\theta = (\frac{p}{4})^d[a]^2$ and by Theorem~ \ref{t2.2}
the following inequality holds
$$
P\Bigl(\sum _{1\le j_1<\ldots<j_d \le n}
a^2 (j_1, \dots , j_d) \varepsilon _{j_1} \ldots \varepsilon _{j_d}
\le 3^d\alpha^{-2d}(d!)^{-2}\theta\Bigr)
\le 6d\exp\Bigl(-\frac{p^{2d}[a]^2}{16^d\delta(a)}\Bigr).
$$
Therefore,
$$
d_{\rm TV}\bigl(Q_{d}(a, X), Q_{d}(a, X)+\eta Z\bigr)\le
12d\exp\Bigl(-\frac{p^{2d}[a]^2}{16^d\delta(a)}\Bigr) +
C(d)\eta^\frac{1}{d} \tfrac{\sqrt{12}}{\alpha \sqrt{p}}
(d!)^{-\frac{1}{d}}
\varkappa^{-\frac{1}{d}}.
$$
Similarly, one has
$$
d_{\rm TV}\bigl(Q_{d}(b, Y), Q_{d}(b, Y)+\eta Z\bigr)
\le
12d\exp\Bigl(-\frac{p^{2d}[b]^2}{16^d\delta(b)}\Bigr) +
C(d)\eta^\frac{1}{d} \tfrac{\sqrt{12}}{\alpha \sqrt{p}}
(d!)^{-\frac{1}{d}}
\varkappa^{-\frac{1}{d}}.
$$
Summing up all the estimates
established above,
we get that
\begin{multline*}
d_{\rm TV}\bigl(Q_{ d}(a, X), Q_{ d}(b, Y)\bigr) \\ \le
C_k\eta^{-k}d_k\bigl(Q_{ d}(a, X), Q_{ d}(b, Y)\bigr)
+ 2C(d)\eta^\frac{1}{d} \tfrac{\sqrt{12}}{\alpha \sqrt{p}}(d!)^{-\frac{1}{d}}\varkappa^{-\frac{1}{d}}
\\+ 12d\exp\Bigl(-\frac{p^{2d}[a]^2}{16^d\delta(a)}\Bigr)
+ 12d\exp\Bigl(-\frac{p^{2d}[b]^2}{16^d\delta(b)}\Bigr).
\end{multline*}
We now take
$\eta =\bigl( \frac{1}{4}d_k(Q_{ d}(a, X), Q_{ d}(b, Y))\bigr)^{\frac{d}{dk+1}}<1$. Then
\begin{multline*}
d_{\rm TV}\bigl(Q_{ d}(a, X), Q_{ d}(b, Y)\bigr)
\\\le
\bigl(4^{\frac{dk}{dk+1}}C_k+ 2C(d)4^{-\frac{1}{dk+1}}(d!)^{-\frac{1}{d}}\tfrac{\sqrt{12}}{\alpha \sqrt{p}}\varkappa^{-\frac{1}{d}}\bigr)
\bigl[d_k\bigl(Q_{ d}(a, X), Q_{ d}(b, Y)\bigr)\bigr]^{\frac{1}{1+kd}}
\\ + 12d\exp\Bigl(-\frac{p^{2d}[a]^2}{16^d\delta(a)}\Bigr)
+ 12d\exp\Bigl(-\frac{p^{2d}[b]^2}{16^d\delta(b)}\Bigr).
\end{multline*}
The theorem is proved.
\end{proof}

\section{Bounds for the total variation distance: general case}\label{sect-4}

We now move on to the case of general polynomials
of the form \eqref{form}. We start with the following techical lemma,
which will be used further in this section.

\begin{lemma}\label{L-1}
Let $d, N, k_*\in \mathbb{N}$, $\omega>0$, $r>0$, $R>0$.
Then there are positive numbers
$c:=c(d,N, \omega, r)$ and
$C:=C(d, N, k_*, \omega, r, R)$, dependent
only on the listed parameters,
such that, for every sequence
$\{X_n\}_{n=1}^\infty$
of independent $N$-dimensional random vectors satisfying the condition
$\mathfrak{D}(\omega, r, R)$, for every
coefficient function collection
$a=(a_0, \ldots, a_d)$, $a_m\colon (\mathbb{N}^3)^m\to \mathbb{R}$,
one has the estimate
$$
P_{\varepsilon, U} \bigl(\mathbb{D}_V \bigl[Q_{d, k_*}(a, X)\bigr] < C[a_d]^2\bigr) \le
6d \exp \Bigl(-\frac{  c [a_d]^2 }{\delta (a_d)} \Bigr)
$$
where $\varepsilon=(\varepsilon_n)_{n=1}^\infty$,
$U=(U_n)_{n=1}^\infty$, $V=(V_n)_{n=1}^\infty$
are collections of independent random variables
from the decomposition \eqref{rep2}
in Theorem \ref{t2.3} and $Q_{d, k_*}(a, X)$ is a polynomial
of the form \eqref{form}.
\end{lemma}

\begin{proof}
Let
$$
X_n \stackrel{Law}{=} \varepsilon_n V_n + (1 - \varepsilon_n)U_n
$$
be a decomposition from Theorem \ref{t2.3}
and
$\varepsilon_n$, $V_n$, $U_n$ are mutually independent.
We note that
\begin{multline*}
Q_{d, k_*}(a, X)
=
\sum_{m=0}^d  \sum _{n_1 < \ldots < n_m}
\sum_{k_1, \ldots, k_m = 1} ^{k_*}\sum_{j_1, \ldots, j_m = 1} ^{N}
a_m((n_1, k_1, j_1), \ldots , (n_m, k_m, j_m))
\\
\times \prod _{i=1}^m
(\varepsilon _{n_i} (V _{n_i, j_i}^{k_i} - \mathbb{E}\bigl[V _{n_i, j_i}^{k_i}\bigr])
+ \bar U_{n_i, k_i, j_i})
\end{multline*}
where
$\bar U_{n, k, j} =
(1 - \varepsilon _{n} ) U_{n, j}^{k} + \varepsilon_{n} \mathbb{E}[V_{n, j}^{k}]
- \mathbb{E} [X_{n, j} ^{k}]$.
Expanding the brackets we get
$$
Q_{d, k_*}(a, X)
=
S_{1}(\varepsilon, V) +
S_{2} (\bar U) + S_{3}
$$
where
$$
S_{1} (\varepsilon, V)
= \sum _{n_1 < \dots < n_d}
\sum_{k_1, \ldots, k_d = 1}^{k_*}
\sum_{j_1, \ldots, j_d = 1}^{N}
a_d((n_1, k_1, j_1), \ldots , (n_d, k_d, j_d))
\prod _{i=1}^d \varepsilon _{n_i}
(V _{n_i, j_i}^{k_i} - \mathbb{E}\bigl[V _{n_i, j_i}^{k_i}\bigr]),
$$
$$
S_{2} (\bar U) =
\sum_{m=0}^d  \sum _{n_1 < \ldots < n_m}
\sum_{k_1, \ldots, k_m = 1} ^{k_*}\sum_{j_1, \ldots, j_m = 1} ^{N}
a_m((n_1, k_1, j_1), \ldots , (n_m, k_m, j_m))
\prod _{i=1} ^m \bar U_{n_i, k_i, j_i},
$$
$$
S_{3} = Q_{d, k_*}(a, X) - S_{1} (\varepsilon, V)
- S_{2} (\bar U).
$$
We now estimate $\mathbb{D}_V \bigl[Q_{d, k_*}(a, X)\bigr]$.
We note that
$\mathbb{E}_V\bigl[S_{1} (\varepsilon, V)\bigr] = \mathbb{E}_V \bigl[S_{3} \bigr]= 0$.
Since $S_{2}(\bar U)$ does not depend on the variables
$V_{n_i, j_i}^{k_i}$, then
$$
\mathbb{D}_V \bigl[Q_{d, k_*}(a, X)\bigr]=
\mathbb{E}_V  \bigl[(S_{1}(\varepsilon, V) + S_{3})^2\bigr]=
\mathbb{E}_V \bigl[S_{1}^2 (\varepsilon, V) \bigr] +
\mathbb{E}_V \bigl[S_{3} ^2\bigr] +
2\mathbb{E}_V\bigl[S_{1} (\varepsilon, V)  S_{3}\bigr].
$$
We now note that the factors of the form
$(V _{n, j}^{k} - \mathbb{E}[V _{n, j}^{k}])$
appear in each summand
of $S_3$ less than $d$ times.
Thus,
$\mathbb{E}_V\bigl[ S_{1} (\varepsilon, V) S_{3}\bigr]  = 0$
since the factors of this form are included in the sum
$S_{1} (\varepsilon, V)$ exactly
$d$ times with different indexes $n$.
Therefore,
$$
\mathbb{D}_V \bigl[Q_{d, k_*}(a, X)\bigr] =
\mathbb{E}_V \bigl[S_{1} ^2 (\varepsilon, V)\bigr] + \E_V \bigl[S_{3} ^2\bigr] \ge
\mathbb{E}_V \bigl[S_{1} ^2 (\varepsilon, V)\bigr].
$$
Theorem \ref{t2.3} implies that
\begin{equation}\label{eq4.1}
\mathbb{E} _V \bigl[S_{1} ^2 (\varepsilon, V) \bigr]
\ge \alpha^d S_0
\end{equation}
where $\alpha=\alpha(r, R, N, k_*)$ and where
$$
S_0 = S_0 (a, \varepsilon) =
\sum _{n_1 < \dots < n_d}
\sum_{k_1, \ldots, k_d = 1} ^{k_*}
\sum_{j_1, \ldots, j_d = 1} ^{N}
a_d^2((n_1, k_1, j_1), \ldots , (n_d, k_d, j_d))
\prod _{i=1}^d\varepsilon^2_{n_i}.
$$
By Theorem \ref{t2.2}, for
$\theta\in\bigl(0, \alpha ^d \bigl(\tfrac{p}{2}\bigr)^d[a_d]^2  \bigr)$,
one has
\begin{multline*}
P_{\varepsilon}
(\alpha^d S_0< \theta)
\\
= P_{\varepsilon} \Bigl(\sum _{n_1 < \dots < n_d}
\sum_{k_1, \ldots, k_d = 1} ^{k_*}
\sum_{j_1, \ldots, j_d = 1} ^{N}
a_d^2((n_1, k_1, j_1), \ldots , (n_d, k_d, j_d))
\prod _{i=1}^d\varepsilon_{n_i}^2 < \frac{\theta}{\alpha^d} \Bigr)  \\
= P_{\varepsilon} \Bigl(\sum _{n_1 < \dots < n_d}
\sum_{k_1, \ldots, k_d = 1} ^{k_*}
\sum_{j_1, \ldots, j_d = 1} ^{N}
a_d^2((n_1, k_1, j_1), \ldots , (n_d, k_d, j_d))
\prod _{i=1}^d\varepsilon_{n_i} < \frac{\theta}{\alpha^d} \Bigr)  \\
\le 6d
\exp \Bigl(-\frac{\theta ^2}{\alpha^{2d}\delta (a_d) [a_d]^2} \Bigr).
\end{multline*}
Substituting
$\ \theta  = \frac{1}{2}\alpha^d  (\frac{p}{2})^{d}[a_d]^2$
we get the estimate
$$
P_{\varepsilon}
(\alpha^d S_0< \theta)
 \le
6d \exp \Bigl(-\frac{\theta ^2}{\alpha^{2d}  \delta (a_d) [a_d]^2} \Bigr)
=
6d \exp \Bigl(-\frac{  (\frac{p}{2})^{2d}[a_d]^2 }{4 \delta (a_d)} \Bigr)
=
6d \exp \Bigl(-\frac{  c [a_d]^2 }{\delta (a_d)} \Bigr)
$$
where
$c =\frac{1}{4}  \bigl(\frac{p}{2}\bigr)^{2d} $.
Hence,
$$
P_{\varepsilon, U} \bigl(\mathbb{D}_V \bigl[Q_{d, k_*}(a, X)\bigr] < C[a_d]^2\bigr) \le
P_{\varepsilon, U}
\bigl(\mathbb{E}_V \bigl[(S_{1} (\varepsilon, V))^2 \bigr] < C[a_d]^2\bigr) \le 6d \exp \Bigl(-\frac{  c [a_d]^2 }{\delta (a_d)} \Bigr)
$$
where
$C = \frac{1}{2} \alpha^d (\frac{p}{2})^{d}.$
The lemma is proved.
\end{proof}

We move on to the first main result of this section.

\begin{theorem} \label{T-2}
Let $k, d, d', N, k_*\in \mathbb{N}, d_* = \max (d, d'), \omega>0, r>0, R>0$.
Then there are numbers
$C:=C(k, d,d', N, k_*, \omega, r, R)$ and $c:=c(d,d', N, k_*, \omega, r)$, dependent
only on the listed parameters,
such that, for every pair of sequences
$\{X_n\}_{n=1}^\infty$ and $\{Y_n\}_{n=1}^\infty$
of independent $N$-dimensional random vectors satisfying the condition
$\mathfrak{D}(\omega, r, R)$, and for every pair of
coefficient function collections
$a=(a_0, \ldots, a_d)$, $a_m\colon (\mathbb{N}^3)^m\to \mathbb{R}$,
$b=(b_0, \ldots, b_{d'})$, $b_m\colon (\mathbb{N}^3)^m\to \mathbb{R}$,
one has the inequality
\begin{multline*}
d_{\rm TV} \bigl(Q_{d, k_*}(a, X), Q_{d', k_*}(b, Y)\bigr) \\
\le C\bigl([a_d]^{-\frac{1}{dk_*}}  + [b_{d'}]^{-\frac{1}{d'k_*}}+ 1\bigr)
\bigl[d_{k}\bigl(Q_{d, k_*}(a, X), Q_{d', k_*}(b, Y)\bigr)\bigr] ^{\frac{1}{1+d_* kk_*}}
\\
+ 12d \exp\Bigl(-\frac{c[a_d]^2 }{\delta (a_d)}\Bigr) +
12d' \exp \Bigl(-\frac{c[b_{d'}]^2 }{\delta (b_{d'})}\Bigr).
\end{multline*}
\end{theorem}

\begin{proof}
Let $\eta \in (0,1), Z \sim \mathcal{N}(0,1)$ and $Z$
is independent from random vectors in sequences $X$ and $Y$.
Then
\begin{multline*}
d_{\rm TV}\bigl(Q_{d, k_*}(a, X), Q_{d', k_*}(b, Y)\bigr)
\le
d_{\rm TV}\bigl(Q_{d, k_*}(a, X), Q_{d, k_*}(a, X)+\eta Z\bigr)
\\
+ d_{\rm TV}\bigl(Q_{d, k_*}(a, X)+\eta Z, Q_{d', k_*}(b, Y)+\eta Z\bigr)
+ d_{\rm TV}\bigl(Q_{d', k_*}(b, Y), Q_{d', k_*}(b, Y)+\eta Z\bigr).
\end{multline*}
By the same argument as in the proof of Theorem \ref{T-1}, we have
$$
d_{\rm TV}\bigl(Q_{d, k_*}(a, X)+\eta Z, Q_{d', k_*}(b, Y)+\eta Z\bigr)
\le C_k\eta^{-k}d_k\bigl(Q_{d, k_*}(a, X), Q_{d', k_*}(b, Y)\bigr).
$$
We now consider the first term above.
By Theorem \ref{t2.3}, without loss of generality,
we can assume that
$$
X_n \stackrel{Law}{=} \varepsilon_n V_n + (1 - \varepsilon_n)U_n
$$
where $\varepsilon_n$, $V_n$, $U_n$ are mutually independent,
$\varepsilon_n$ are Bernoulli random variables
with the probability of success $p = p(\omega, r, N)$,
$V_n$ are logarithmically concave random vectors,
and $U_n$ are some other random vectors.
Let $\varphi\in C_0^\infty(\mathbb{R})$,
$\|\varphi\|_{\infty} \le 1$.
Theorem~\ref{t2.1} implies that
$$
\mathbb{E}_V \bigl[| \varphi (Q_{d, k_*}(a, X)) -  \varphi (Q_{d, k_*}(a, X) + \eta Z)|\bigr] \le
\frac {C_1(d, k_*)}
{(\mathbb{D}_V\bigl[Q_{d, k_*}(a, X)\bigr])^\frac{1}{2d k_*}}
|\eta Z |^\frac{1}{d k_*}.
$$
Let
$C:=\min\{C(d, N, k_*, \omega, r, R),
C(d', N, k_*, \omega, r, R)\}$,
where the constants $C(d, N, k_*, \omega, r, R)$ and
$C(d', N, k_*, \omega, r, R)$ are from
Lemma \ref{L-1}.
For $\theta = C[a_d]^2$ we have
\begin{multline*}
\mathbb{E} \bigl[|\varphi (Q_{d, k_*}(a, X)) - \varphi (Q_{d, k_*}(a, X)+ \eta Z) | \bigr]
\\
=   \mathbb{E}_{Z} \mathbb{E} _{\varepsilon, U} \E_V \bigl[|
\varphi (Q_{d, k_*}(a, X)) - \varphi (Q_{d, k_*}(a, X)+ \eta Z) |\bigr]
\\
= \mathbb{E}_{Z} \E _{\varepsilon, U} \mathbb{E}_V \Bigl[ |
(\varphi (Q_{d, k_*}(a, X)) - \varphi (Q_{d, k_*}(a, X)+ \eta Z ) |
( I_{\{ \mathbb{D}_V [Q_{d, k_*}(a, X)] \ge \theta \}} +
I_{\{ \mathbb{D}_V [Q_{d, k_*}(a, X)] < \theta \}}) \Bigr]
\\ \le \mathbb{E}_{Z} \Bigl[\frac {C_1(d,k_*)}{\theta^\frac{1}{2d k_*}} |\eta Z|^\frac{1}{d k_*}
\Bigr] +
2\| \varphi \|_\infty \mathbb{E}_{Z} \bigl[  \mathbb{E}_{\varepsilon , U}
    [I_{\{ \mathbb{D}_V [Q_{d, k_*}(a, X)] < \theta \}}] \bigr]
\\
\le \frac {C_1(d,k_*)}{\theta ^ \frac{1}{2d k_*}}
|\eta| ^ \frac{1}{d k_*}
+ 2P_{\varepsilon, U} \bigl(\mathbb{D}_V \bigl[Q_{d, k_*}(a, X)\bigr] < \theta\bigr) .
\end{multline*}
By Lemma \ref{L-1}
$$
P_{\varepsilon, U} (\mathbb{D}_V \bigl[Q_{d, k_*}(a, X)\bigr] < \theta)
\le
 6d \exp \Bigl(-\frac{  c [a_d]^2 }{\delta (a_d)} \Bigr)
$$
where
$c = \min\{c(d,N, \omega, r), c(d',N, \omega, r)\}$
and $c(d,N, \omega, r)$, $c(d',N, \omega, r)$
are the constants from Lemma \ref{L-1}.
Then
we get the estimate
$$
\E \bigl[|\varphi (Q_{d, k_*}(a, X)) - \varphi (Q_{d, k_*}(a, X)+ \eta Z)| \bigr]
\le
C_2  [a_d]^{-\frac{1}{dk_*}}
\eta^{\frac{1}{dk_*}}+
12d \exp \Bigl(-\frac{ c[a_d]^2 }{ \delta (a_d)} \Bigr)
$$
where $C_2:=C_2(d, d', N, k_*, \omega, r, R)$.
Similarly,
$$
\E \bigl[|\varphi (Q_{d', k_*}(b, Y)) - \varphi (Q_{d', k_*}(b, Y)+ \eta Z)| \bigr]
\le
C_2  [b_{d'}]^{-\frac{1}{d'k_*}}
\eta^{\frac{1}{d'k_*}}+
12d' \exp \Bigl(-\frac{ c[b_{d'}]^2 }{ \delta (b_{d'})} \Bigr).
$$
Thus,
\begin{multline*}
d_{\rm TV}\bigl(Q_{d, k_*}(a, X), Q_{d', k_*}(b, Y)\bigr)\le
C_2
[a_d]^{-\frac{1}{dk_*}}
\eta^{\frac{1}{dk_*}}+
12d \exp \Bigr(-\frac{  c[a_d]^2 }{\delta(a_d)} \Bigl)
\\+
C_2 [b_{d'}]^{-\frac{1}{d'k_*}}
\eta ^{\frac{1}{d'k*}}+
12d' \exp \Bigl(-\frac{c[b_{d'}]^2 }{\delta (b_{d'})} \Bigr) +
C_k\eta^{-k} d_{k}(Q_{d, k_*}(a, X), Q_{d', k_*}(b, Y))
\end{multline*}
where  $c = c(d, d', N, k_*, \omega, r)$, $ C_2 = C_2 (d,d', N, k_*, \omega, r, R) $.
Let now
$$
\eta =
\Bigl(\frac{1}{4} d_{k} (Q_{d, k_*}(a, X), Q_{d', k_*}(b, Y)) \Bigr)^{\tfrac{d_*k_*}{d_* kk_*+1}} < 1.
$$
Then we get the estimate
\begin{multline*}
d_{\rm TV}\bigl(Q_{d, k_*}(a, X), Q_{d', k_*}(b, Y)\bigr)
\\
\le
C_3
\bigl([a_d]^{-\frac{1}{dk_*}}  + [b_{d'}]^{-\frac{1}{d'k_*}} + 1\bigr)
\bigl[d_{k}(Q_{d, k_*}(a, X), Q_{d', k_*}(b, Y))\bigr]^{\frac{1}{1+d_* kk_*}}
\\+
12d \exp \Bigr(-\frac{  c[a_d]^2 }{\delta(a_d)} \Bigl)
+
12d' \exp \Bigl(-\frac{c[b_{d'}]^2 }{\delta (b_{d'})} \Bigr)
\end{multline*}
where $C_3 = C_3 (d,d', N, k_*, \omega, r, R, k)$.
The theorem is proved.
\end{proof}

We note that in the previous theorem
we do not impose any assumptions on the moments
of vectors $X_n$ and $Y_n$. In the next theorem,
under a certain moment assumption,
we show that one can use
influence factor $\delta$
of not only the highest order monomials.
This allows to improve the exponent of the
$d_k$ distance in the resulting estimate, similarly to Theorem \ref{T-BC}.

\begin{theorem}\label{T-3}
Let $k, d, d', l, l', N, k_*\in \mathbb{N}, \omega>0, r>0, R>0$, $l'<d', l<d$.
There are numbers
$C:=C(k, d,d', N, k_*, \omega, r, R)$ and $c:=c(d,d', N, k_*, \omega, r)$, depending only on the listed parameters, such that, for any pair of sequences
$\{X_n\}_{n=1}^\infty$ and $\{Y_n\}_{n=1}^\infty$
of independent $N$-dimensional random vectors for which the condition
$\mathfrak{D}(\omega, r, R)$ is satisfied, and  for each pair of coefficient function collections
$a=(a_0, \ldots, a_d)$, $a_m\colon (\mathbb{N}^3)^m\to \mathbb{R}$,
$b=(b_0, \ldots, b_{d'})$, $b_m\colon (\mathbb{N}^3)^m\to \mathbb{R}$,
one has the bound
\begin{multline*}
d_{\rm TV}\bigl(Q_{d, k_*}(a, X), Q_{d', k_*}(b, Y)\bigr)
\le
C
\bigl([a_l]^{-\frac{1}{lk_*}}  + [b_{l'}]^{-\frac{1}{l'k_*}} + 1\bigr)
\Bigl(
\bigl[d_{k}(Q_{d, k_*}(a, X), Q_{d', k_*}(b, Y))\bigr]^{\frac{1}{1+l_* kk_*}}
\\+
(1+M_{2k_*}) ^\frac{d}{l}[a_{l+1,d}]^{\frac{1}{lk_*}}
+
(1+M_{2k_*})^\frac{d'}{l'}[b_{l'+1,d'}]^{\frac{1}{l'k_*}}
\Bigr)
+
36l
\exp \Bigl(-\frac{c[a_l]^2}{\delta (a_l)} \Bigr)
+
36l'
\exp \Bigl(-\frac{  c[b_{l'}]^2}{\delta (b_{l'})} \Bigr)
\end{multline*}
where
$l_* = \max(l, l')$
and
$$
M_q:=\max\bigl\{\max_{1\le j\le N}\sup_{n \in \mathbb{N}}
[\mathbb{E}|X_{n,j}|^q]^{1/q},
\max_{1\le j\le N}\sup_{n \in \mathbb{N}}
[\mathbb{E}|Y_{n,j}|^q]^{1/q}\bigr\},\quad q\in[1, +\infty).
$$
\end{theorem}
\begin{proof}
Let $\varepsilon_n$, $V_n$, $U_n$, $p$, $\alpha$, $Z$, $\eta$
be the same as in the proof of the previous theorem
and Lemma \ref{L-1}.
Let
$$
P(a,X) = Q_{d, k_*}(a,X) - Q_{l, k_*}(a,X).
$$
We note that
\begin{multline*}
d_{\rm TV}\bigl(Q_{d, k_*}(a, X), Q_{d', k_*}(b, Y)\bigr)
\le
d_{\rm TV}\bigl(Q_{d, k_*}(a, X), Q_{d, k_*}(a, X)+\eta Z\bigr)
\\
+ d_{\rm TV}\bigl(Q_{d, k_*}(a, X)+\eta Z, Q_{d', k_*}(b, Y)+\eta Z\bigr)
+ d_{\rm TV}\bigl(Q_{d', k_*}(b, Y), Q_{d', k_*}(b, Y)+\eta Z\bigr)
\end{multline*}
and
\begin{multline*}
d_{\rm TV}\bigl(Q_{d, k_*}(a, X), Q_{d, k_*}(a, X)+\eta Z\bigr) \le
d_{\rm TV}\bigl(Q_{l, k_*}(a, X)+\eta Z, Q_{d, k_*}(a, X)+\eta Z\bigr) \\ +
d_{\rm TV}\bigl(Q_{l, k_*}(a, X), Q_{l, k_*}(a, X)+\eta Z\bigr) +
d_{\rm TV}\bigl(Q_{d, k_*}(a, X), Q_{l, k_*}(a, X)\bigr).
\end{multline*}
Let
$C:=\min\{C(l, N, k_*, \omega, r, R),
C(l', N, k_*, \omega, r, R)\}$
where the constants $C(l, N, k_*, \omega, r, R)$ and
$C(l', N, k_*, \omega, r, R)$ are from
Lemma \ref{L-1}.
Proceeding as in the proof of the previous theorem and using Lemma \ref{L-1}, for $\theta = C[a_l]^2$, we obtain the estimate
\begin{multline*}
\mathbb{E} \bigl[|\varphi (Q_{l, k_*}(a, X)) - \varphi (Q_{l, k_*}(a, X)+ \eta Z) | \bigr]
\\
\le
\mathbb{E}_Z \Bigl[ \frac {C_1(l,k_*)}{\theta ^ \frac{1}{2l k_*}} |\eta Z| ^ \frac{1}{l k_*} \Bigr]
+ 2P_{\varepsilon, U} (\mathbb{D}_V \bigl[Q_{l, k_*}(a, X)\bigr] < \theta)
\le
 \frac {C_2}{[a_l] ^ \frac{1}{l k_*}} \eta  ^ \frac{1}{l k_*}  +
12l
\exp \Bigl(-\frac{c [a_l]^2}{\delta (a_l)} \Bigr)
\end{multline*}
where
$c = \min\{c(l,N, \omega, r), c(l',N, \omega, r)\}$, and
$c(l,N, \omega, r)$, $c(l',N, \omega, r)$
are the constants from Lemma \ref{L-1}, and where $C_2:=C_2(l, l', N, k_*, \omega, r, R)$.

We now estimate the distance
$$
d_{\rm TV}\bigl(Q_{l, k_*}(a, X)+\eta Z, Q_{d, k_*}(a, X)+\eta Z\bigr).
$$
Using Theorem $\ref{t2.1.1}$ and Remark \ref{rm2.2},  for $\theta = C[a_l]^2$, we get the bound
\begin{multline*}
\mathbb{E} \bigl[ \varphi (Q_{l, k_*}(a, X) + \eta Z) -  \varphi (Q_{d, k_*}(a, X) + \eta Z) \bigr] \\
=
\mathbb{E}_Z \mathbb{E}_{\varepsilon, U} \mathbb{E}_V \bigl[| \varphi (Q_{l, k_*}(a, X) + \eta Z) -  \varphi (Q_{d, k_*}(a, X) + \eta Z) | \bigr] \\
\le
\mathbb{E}_Z \mathbb{E}_{\varepsilon, U} \mathbb{E}_V \bigl[| \varphi (Q_{l, k_*}(a, X) + \eta Z) -  \varphi (Q_{d, k_*}(a, X) + \eta Z) | (I_{\mathbb{D}_V Q_{l, k _*} (a, X)\ge \theta} + I_{\mathbb{D}_V Q_{l, k _*} (a, X) < \theta}) \bigr] \\
\le
\mathbb{E}_Z \mathbb{E}_{\varepsilon, U} \bigl[ \frac{C(l, d, k_*)}{\theta ^{\frac{1}{2lk_*}}} \bigl( \mathbb{E}_V [P^2(a, X)] \bigr) ^ {\frac{1}{2lk_*}} \bigr]  +
2 P_{Z, \varepsilon, U} \bigl(\mathbb{D}_V [Q_{l, k_*} (a, X)] < \theta\bigr) \\
=
\frac{C(l, d, k_*)}{\theta ^{\frac{1}{2lk_*}}}
\mathbb{E}_{\varepsilon, U} \bigl[ \bigl( \mathbb{E}_V [P^2(a, X)] \bigr) ^ {\frac{1}{2lk_*}} \bigr]  +
2 P_{ \varepsilon, U} \bigl(\mathbb{D}_V [Q_{l, k_*} (a, X)] < \theta\bigr).
\end{multline*}
Using H\"older's inequality, we get
$$\mathbb{E}_{\varepsilon, U}\bigl[ \bigl( \mathbb{E}_V [P^2(a, X)] \bigr) ^ {\frac{1}{2lk_*}} \bigr] \le \bigl(\mathbb{E} [P^2(a, X)] \bigr) ^{\frac{1}{2lk_*}}. $$
Hence, applying Lemma \ref{L-1},
\begin{multline*}
\mathbb{E} \bigl[ \varphi (Q_{l, k_*}(a, X) + \eta Z) -  \varphi (Q_{d, k_*}(a, X) + \eta Z) \bigr] \\ \le
\frac{C(l, d, k_*)}{\theta ^{\frac{1}{2lk_*}}}
 \bigl(\mathbb{E} [P^2(a, X)] \bigr) ^{\frac{1}{2lk_*}}  +
2 P_{ \varepsilon, U} \bigl(\mathbb{D}_V [Q_{l, k_*} (a, X)] < \theta\bigr)
\\ \le
\frac{C(l, d, k_*)}{\theta ^{\frac{1}{2lk_*}}}
\bigl(\mathbb{E} [P^2(a, X)] \bigr) ^{\frac{1}{2lk_*}} +
12l
\exp \Bigl(-\frac{c[a_l]^2 }{\delta (a_l) } \Bigr).
\end{multline*}
Similarly, we have
$$
\mathbb{E} \bigl[ \varphi (Q_{l, k_*}(a, X) ) -  \varphi (Q_{d, k_*}(a, X) ) \bigr]
 \le
\frac{C(l, d, k_*)}{\theta ^{\frac{1}{2lk_*}}}
\bigl(\mathbb{E} [P^2(a, X)] \bigr) ^{\frac{1}{2lk_*}} +
12l
\exp \Bigl(-\frac{c[a_l]^2 }{\delta (a_l) } \Bigr).
$$
Since
\begin{multline*}
P(a,X)\\
=
\sum_{m=l+1}^d  \sum _{n_1 < \ldots < n_m}
\sum_{k_1, \ldots, k_m = 1} ^{k_*}\sum_{j_1, \ldots, j_m = 1} ^{N}
a_m((n_1, k_1, j_1), \ldots , (n_m, k_m, j_m))
\prod _{i=1} ^m
(X_{n_i, j_i}^{k_i}- \mathbb{E}[X_{n_i, j_i}^{k_i}]),
\end{multline*}
we get that
\begin{multline*}
\bigl(\mathbb{E} [P^2(a,X)]\bigr)^{\frac{1}{2}} \le  \\
\sum_{m=l+1}^d\sum_{k_1, \ldots, k_m = 1} ^{k_*}\sum_{j_1, \ldots, j_m = 1} ^{N}
\Bigl(\mathbb{E}\Bigl[\Bigr(\!\!\!\!\!\sum _{n_1 < \ldots < n_m}
\!\!\!a_m((n_1, k_1, j_1), \ldots , (n_m, k_m, j_m)) \prod _{i=1} ^m
(X_{n_i, j_i}^{k_i}- \mathbb{E}[X_{n_i, j_i}^{k_i}])\Bigl)^2\Bigr]\Bigr)^{\frac{1}{2}}
\\=
\sum_{m=l+1}^d
\sum_{k_1, \ldots, k_m = 1} ^{k_*}\sum_{j_1, \ldots, j_m = 1} ^{N}
\Bigl(\sum _{n_1 < \ldots < n_m}a_m^2((n_1, k_1, j_1), \ldots , (n_m, k_m, j_m))
\prod_{i=1}^m
\mathbb{D}\bigl[X_{n_i, j_i}^{k_i}\bigr]\Bigr)^{\frac{1}{2}}
\\ \le
C(d, k_*, N)
\Bigl(\!\!\sum_{m=l+1}^d
\sum_{k_1, \ldots, k_m = 1} ^{k_*}\sum_{j_1, \ldots, j_m = 1} ^{N}
\sum _{n_1 < \ldots < n_m}\!\!\!a_m^2((n_1, k_1, j_1), \ldots , (n_m, k_m, j_m))
\prod_{i=1}^m
\mathbb{D}\bigl[X_{n_i, j_i}^{k_i}\bigr]\!\Bigr)^{\frac{1}{2}}.
\end{multline*}
Since
$$
\mathbb{D}\bigl[X_{n_i, j_i}^{k_i}\bigr]
\le\mathbb{E}[X_{n_i, j_i}^{2k_i}]
\le \bigl(\mathbb{E}[X_{n_i, j_i}^{2k_*}]
\bigr)^{\frac{2k_i}{2k_*}}
\le M_{2k_*}^{2k_i}\le
(1+M_{2k_*})^{2k_i}\le
(1+M_{2k_*})^{2k_*},
$$
we get
\begin{multline*}
\mathbb{E} [P^2(a,X)]
\\
\le
C^2(d, k_*, N)(1+M_{2k_*})^{2dk_*}
\sum_{m=l+1}^d  \sum _{n_1 < \ldots < n_m}
\sum_{k_1, \ldots, k_m = 1} ^{k_*}\sum_{j_1, \ldots, j_m = 1} ^{N}
a_m^2((n_1, k_1, j_1), \ldots , (n_m, k_m, j_m))
\\=
C^2(d, k_*, N)(1+M_{2k_*})^{2dk_*}[a_{l+1,d}]^2.
\end{multline*}
%Substituting
%$\ \theta  = \frac{1}{2}\alpha^l  (\frac{p}{2})^{lk_*}[a_l]^2$,
Then we get
\begin{multline*}
d_{\rm TV}\bigl(Q_{d, k_*}(a, X), Q_{d, k_*}(a, X)+\eta Z\bigr)  \\
\le
 \frac {C_2}{[a_l] ^ \frac{1}{l k_*}} \eta  ^ \frac{1}{l k_*} +
2\frac{C(l, d, k_*)}{\theta ^{\frac{1}{2lk_*}}} \bigl(C^2(d, k_*, N)(1+M_{2k_*})^{2dk_*}[a_{l+1,d}]^2\bigr)^{\frac{1}{2lk_*}}  +
36l
\exp \Bigl(-\frac{c [a_l]^2 }{\delta (a_l) } \Bigr) \\
=
\frac {C_2}{ [a_l] ^ \frac{1}{l k_*}} \eta  ^ \frac{1}{l k_*} +
\frac{C_3}{[a_l] ^{\frac{1}{lk_*}}} (1+M_{2k_*}) ^\frac{d}{l}[a_{l+1,d}]^{\frac{1}{lk_*}}  +
36l
\exp \Bigl(-\frac{  c [a_l]^2}{\delta (a_l)} \Bigr)
\end{multline*}
where $C_3:=C_3(d,d', N, k_*, \omega, r, R)$.
Similarly,
\begin{multline*}
\mathbb{E} \bigl[|\varphi (Q_{l', k_*}(b, Y)) - \varphi (Q_{l', k_*}(b, Y)+ \eta Z) | \bigr]
\\ \le
\frac {C_2}{[b_{l'}] ^ \frac{1}{l' k_*}} \eta  ^ \frac{1}{l' k_*} +
\frac{C_3}{[b_{l'}] ^{\frac{1}{l'k_*}}} (1+M_{2k_*}) ^\frac{d'}{l'}[b_{l'+1,d'}]^{\frac{1}{l'k_*}}  +
36l'
\exp \Bigl(-\frac{ c [b_{l'}]^2}{\delta (b_{l'})} \Bigr).
\end{multline*}
Thus, we get
\begin{multline*}
d_{\rm TV}\bigl(Q_{d, k_*}(a, X), Q_{d', k_*}(b, Y)\bigr)\le
C_k \frac{1}{\eta^k} d_{k}(Q_{d, k_*}(a, X), Q_{d', k_*}(b, Y))  \\ +
\frac {C_2}{ [a_l] ^ \frac{1}{l k_*}} \eta  ^ \frac{1}{l k_*} +
\frac{C_3}{[a_l] ^{\frac{1}{lk_*}}} (1+M_{2k_*}) ^\frac{d}{l}[a_{l+1,d}]^{\frac{1}{lk_*}}  +
36l
\exp \Bigl(-\frac{  c [a_l]^2}{\delta (a_l)} \Bigr) \\
+
\frac {C_2}{[b_{l'}] ^ \frac{1}{l' k_*}} \eta  ^ \frac{1}{l' k_*} +
\frac{C_3}{[b_{l'}] ^{\frac{1}{l'k_*}}} (1+M_{2k_*}) ^\frac{d'}{l'}[b_{l'+1,d'}]^{\frac{1}{l'k_*}}  +
36l'
\exp \Bigl(-\frac{ c [b_{l'}]^2}{\delta (b_{l'})} \Bigr).
\end{multline*}
Let now
$$\displaystyle
\eta =
\Bigl(\frac{1}{4} d_{k} (Q_{d, k_*}(a, X), Q_{d', k_*}(b, Y)) \Bigr)^{\tfrac{l_*k_*}{l_* kk_*+1}} < 1.$$
Then we get
\begin{multline*}
d_{\rm TV}\bigl(Q_{d, k_*}(a, X), Q_{d', k_*}(b, Y)\bigr)
\le
C_4
\bigl([a_l]^{-\frac{1}{lk_*}}  + [b_{l'}]^{-\frac{1}{l'k_*}} + 1\bigr)
\bigl[d_{k}(Q_{d, k_*}(a, X), Q_{d', k_*}(b, Y))\bigr]^{\frac{1}{1+l_* kk_*}}
\\+
\frac{C_3}{[a_l] ^{\frac{1}{lk_*}}} (1+M_{2k_*}) ^\frac{d}{l}[a_{l+1,d}]^{\frac{1}{lk_*}}  +
36l
\exp \Bigl(-\frac{c[a_l]^2}{\delta (a_l)} \Bigr) \\
+
\frac{C_3}{[b_{l'}] ^{\frac{1}{l'k_*}}} (1+M_{2k_*}) ^\frac{d'}{l'}[b_{l'+1,d'}]^{\frac{1}{l'k_*}}  +
36l'
\exp \Bigl(-\frac{  c[b_{l'}]^2}{\delta (b_{l'})} \Bigr),
\end{multline*}
$C_4 = C_4 (d,d', l, l', N, k_*, \omega, r, R, k).$
This estimate implies the announced bound.
The theorem is proved.
\end{proof}

As an application of Theorem \ref{T-2}
we provide a sharper version
of the invariance principle from
\cite{BC19} (see Theorem 3.10 there).

\begin{theorem}\label{T-4}
Let $k, d, N, k_*\in \mathbb{N}$,
$\omega>0, r>0, R>0$, $M>0$ and let
$\{X_n\}_{n=1}^\infty$ be a sequence
of independent $N$-dimensional random vectors such that the condition
$\mathfrak{D}(\omega, r, R)$ is satisfied.
Assume that
$$
\max\bigl\{1, \sup\{[\mathbb{E} |X_{n, i}|^{3k_*}]^{1/(3k_*)}\colon
i \in \{1, \ldots, N \}, n\in\mathbb{N}\bigr\}
\le M.
$$
Let $\{Z_n\}_{n=1}^\infty$ be a sequence
of independent random vectors in
$\mathbb{R}^{N\times k_*} $
such that
$$
Z_n=(X_{n,1} - \mathbb{E}X_{n,1}, \ldots,
X_{n, N} - \mathbb{E}X_{n , N},\ldots,
X_{n, 1}^{k_*} - \mathbb{E}X_{n, 1}^{k_*},\ldots,
X_{n, N}^{k_*} - \mathbb{E}X_{n, N}^{k_*}).
$$
Let
$ \{ G_n \}_{n=1} ^ \infty $ be a sequence of centered independent Gaussian random vectors in
$ \mathbb{R}^{N\times k_*} $.
Assume that this sequence is independent from the sequence  $\{X_n\}_{n=1}^\infty$ and each $G_n$ has
the same covariance matrix like the vector
$Z_n$, i.e.
$Cov (G_n) = Cov (Z_n)$.
Then, for each collection of coefficient functions
$a=(a_1, \ldots, a_d)$, $a_m\colon (\mathbb{N}^3)^m\to \mathbb{R}$,
there is a number
$ C = C(k_*, d, N, \varepsilon, r, R, M)$
such
that
$$
d_{\rm TV} (Q_{d, k_*}(a, X), S_d(a,G)) \le
C \bigl([a_d]^{-\frac{1}{dk_*}} + 1\bigr)
\bigl([a]^{\frac{2}{1+3dk_*}}+1\bigr)
\bigl[\sqrt{\delta (a)}\bigr] ^{\frac{1}{1+3dk_*}}
$$
where
\begin{multline*}
S_d(a,G) \\
= \sum_{m=0}^d  \sum _{n_1 < \dots < n_m}   \sum _{k_1 = 1, \ldots, k_m = 1} ^{k_*}
\sum_{j_1 = 1, \dots, j_m = 1}^N a_m((n_1, k_1, j_1), \ldots , (n_m, k_m, j_m))
\prod _{i=1} ^mG_{n_i, (k_i-1)\cdot N + j_i}.
\end{multline*}
\end{theorem}

\begin{proof}
We note that $M_3(Z)\le M$ and
from Theorem \ref{t2.4}
we get the estimate
$$
d_{3} (Q_{d, k_*}(a, X), S_d (a,G))
\\
\le C(d, N, M) [a]^2
\sqrt{\delta (a)}.
$$
Theorem \ref{T-2}
implies the bound
\begin{multline*}
d_{\rm TV} (Q_{d, k_*}(a, X), S_d(a,G))
\\
\le C\bigl(2[a_d]^{-\frac{1}{dk_*}} + 1\bigr)
\bigl[d_{3}(Q_{d, k_*}(a, X), S_d(a,G))\bigr] ^{\frac{1}{1+3dk_*}}
+ 24d \exp\Bigl(-\frac{c[a_d]^2 }{\delta (a_d)}\Bigr)
\end{multline*}
with some constants
$C:=C(k, d, N, k_*, \omega, r, R)$ and
$c:=c(d, N, k_*, \omega, r)$.
Thus,
\begin{multline*}
d_{\rm TV} (Q_{d, k_*}(a, X), S_d(a,G))
\le
C_0\bigl(2[a_d]^{-\frac{1}{dk_*}} + 1\bigr)
[a]^{\frac{2}{1+3dk_*}} \delta (a) ^{\frac{1}{2+6dk_*}}
+ 24d \exp\Bigl(-\frac{c[a_d]^2 }{\delta (a_d)}\Bigr)
\\
\le
C_0\bigl(2[a_d]^{-\frac{1}{dk_*}} + 1\bigr)
[a]^{\frac{2}{1+3dk_*}} \delta (a) ^{\frac{1}{2+6dk_*}}
+ 24d \frac{\delta (a_d)}{c[a_d]^2}
\\
\le
C_0\bigl(2[a_d]^{-\frac{1}{dk_*}} + 1\bigr)
[a_d]^{\frac{2}{1+3dk_*}} \delta (a_d) ^{\frac{1}{2+6dk_*}}
+ 24dc^{-1} \delta (a_d) ^{\frac{1}{2+6dk_*}}
[a_d]^{-\frac{1}{1+3dk_*}}
\\ \le
C_1\bigl([a_d]^{-\frac{1}{dk_*}} + 1\bigr)
\bigl([a]^{\frac{2}{1+3dk_*}}+1\bigr)
\delta (a) ^{\frac{1}{2+6dk_*}}
\end{multline*}
where
$C_0= C\cdot \bigl[C(d, N, M)\bigr]^{\frac{1}{1+3dk_*}}$
and where
$C_1 = C_1(k_*, d, N, \varepsilon, r, R, M)$.
The theorem is proved.
\end{proof}

\begin{remark}
{\rm
We note that, in the proof above,
instead of Theorem~\ref{T-2},
we could have applied Theorem~\ref{T-3}.
In that case, under the same assumptions as
in Theorem~\ref{T-4},
one obtains the estimates similar to the result above.
}
\end{remark}

\section{Bounds for Fourier transforms}
\label{sect-5}

In this section we will study upper bounds for the Fourier transforms of the distributions
of random variables \eqref{form} (i.e. characteristic functions of these random variables).

In \cite{GPU} the following result was obtained (see Theorem $3$ there).

\begin{theorem}\label{T-GPU}
Let $f_0(x):=x_1^d+\ldots+x_N^d$ and let
$\{X_n:=(X_{n, 1}, \ldots, X_{n, N})\}_{n=1}^\infty$
be a sequence
of independent and identically distributed $N$-dimensional random vectors.
Assume that for some $n_0$ the distribution of the random vector
$X_1+\ldots+X_{n_0}$ has an absolutely continuous component.
Let $\omega, r, R$ be such that $X_1+\ldots+X_{n_0}\in \mathfrak{D}(\omega, r, R)$.
Then for every $a\in \mathbb{R}^N$ and for every $n\ge 6n_0/\omega$, one has
$$
\bigl|\mathbb{E} \bigl[\exp(itf_0(S_n+a))\bigr]\bigr|
\le 2\exp\Bigl(\frac{-n\omega^2}{24n_0}\Bigr)
+C(N)r^{-N}(n_0\omega^{-1})^{N/2}|t|^{-N/d}
$$
where $S_n:=X_1+\ldots+X_n$ and
$C(N)$ is some constant dependent only on $N$.
\end{theorem}

Our goal in this section is to obtain bounds
for random variables of the form \eqref{form}
similar to the one from Theorem \ref{T-GPU}.

We firstly note that there is a general
observation connecting the behaviour of the
total variation distance between the
distribution of a random variable
and the distribution of its shift
with
the estimates for the Fourier
transform of that distribution.

\begin{lemma} \label{L-3}
Let $\xi$ be a random variable. Then, for any function
$\varphi \in C_b^{\infty}(\mathbb{R})$
and for any $s>0$, one has
$$
\mathbb{E} [\varphi' (\xi)] \le
\| \varphi' \|_{\infty} D_\xi(s) +
\frac{2}{s} \| \varphi \|_\infty
$$
where
$$
D_\xi(s):=
\sup\limits_{|h|\le s}
d_{\rm{TV}}(\xi, \xi +h).
$$
\end{lemma}
\begin{proof}
We note that, for $h \in \mathbb{R}$, one has
$$
\mathbb{E}[\varphi '(\xi) - \varphi ' (\xi +h)] \le \| \varphi ' \|_\infty d_{\rm{TV}}(\xi, \xi +h)
\le  \| \varphi ' \|_\infty D_\xi (|h|).
$$
Thus, for $s \in (0, \infty)$ and for a random variable U uniformly distributed on $[0, 1]$ and independent
from $\xi$, we get that
\begin{multline*}
\mathbb{E}[\varphi ' (\xi)] =
\mathbb{E}_{\xi, U}[\varphi ' (\xi)] =
\mathbb{E}_{\xi, U}[\varphi ' (\xi ) - \varphi ' (\xi +sU) + \varphi ' (\xi +sU)]  \\
\le
\| \varphi ' \|_{\infty}\mathbb{E}_U [D_\xi (sU)] +
\mathbb{E}_{\xi, U} [\varphi ' (\xi +sU)]
\le
\| \varphi ' \|_{\infty}D_\xi (s) +
\mathbb{E}_{\xi, U} [\varphi ' (\xi +sU)].
\end{multline*}
We now consider the last summand
$$ \displaystyle
\mathbb{E}_{U} [\varphi ' (\xi +sU)] =
\int _0 ^1 \varphi ' (\xi + su) du =
\frac{1}{s} \Bigl(\varphi (\xi + s) - \varphi (\xi )
\Bigr) \le
\frac{2}{s} \| \varphi \|_\infty.
$$
Then
$$
\displaystyle
\mathbb{E}[\varphi ' (\xi)] \le
\| \varphi ' \|_{\infty}D_\xi (s) +
\frac{2}{s} \| \varphi \|_\infty.
$$
The lemma is proved.
\end{proof}

\begin{corollary}
Let $d, N, k_*\in \mathbb{N},\omega>0, r>0, R>0$.
Then there are positive numbers $c:=c(d,N, k_*, \omega, r)$ and
$C:=C(d, N, k_*, \omega, r, R)$,
dependent only on the listed parameters, such that, for every sequence
$\{X_n\}_{n=1}^\infty$ of
independent $N$-dimensional random vectors satisfying the condition
$\mathfrak{D}(\omega, r, R)$,
for every coefficient function collection
$a=(a_1, \ldots, a_d)$, $a_m\colon (\mathbb{N}^3)^m\to \mathbb{R}$,
for every function
$\varphi \in C_b^{\infty}(\mathbb{R})$, and for every $s>0$,
one has
$$
\mathbb{E}[\varphi ' ( Q_{d, k_*}(a, X))] \le
\| \varphi ' \|_{\infty} \Bigl(C\bigl([a_d]^{-\frac{1}{dk_*}}+ 1\bigr) s^{\frac{1}{1+dk_*}} + 24d \exp \Bigl(-\frac{c[a_d]^2}{\delta(a_d)} \Bigr)\Bigr) +
\frac{2}{s} \| \varphi \|_\infty
$$
for $Q_{d, k_*}(a, X)$ of the form \eqref{form}.
\end{corollary}
\begin{proof}
We note that, for any function
$\varphi_0 \in C_0 ^ \infty (\mathbb{R})$ and for any $h \in \mathbb{R}$,
$$
\mathbb{E} [ \varphi_0(Q_{d, k_*}(a, X)) - \varphi_0(Q_{d, k_*}(a, X) + h) ] \le
 \| \varphi_0 ' \|_\infty |h|.
$$
Thus,
$$
d_1 (Q_{d, k_*}(a, X), Q_{d, k_*}(a, X) + h) \le |h| .
$$
Then, by the Theorem
\ref{T-2} for $k=1$ and $h \in \mathbb{R}$,
we have
\begin{multline*}
d_{\rm{TV}} ( Q_{d, k_*}(a, X), Q_{d, k_*}(a, X) + h)
\\
\le
C\bigl([a_d]^{-\frac{1}{dk_*}}+ 1\bigr) \bigl(d_1 (Q_{d, k_*}(a, X), Q_{d, k_*}(a, X) + h) \bigr)^{\frac{1}{1+dk_*}} + 24d \exp \Bigl(-\frac{c[a_d]^2}{\delta(a_d)} \Bigr)
\\ \le
C\bigl([a_d]^{-\frac{1}{dk_*}}+ 1\bigr)
|h|^{\frac{1}{1+dk_*}} + 24d \exp \Bigl(-\frac{c[a_d]^2}{\delta(a_d)} \Bigr) .
\end{multline*}
Hence,
$$
D_{Q_{d, k_*}(a, X)} (s) \le
C\bigl([a_d]^{-\frac{1}{dk_*}}+ 1\bigr) s^{\frac{1}{1+dk_*}} + 24d \exp \Bigl(-\frac{c[a_d]^2}{\delta(a_d)} \Bigr).
$$
We now use Lemma \ref{L-3} and obtain
the announced bound
$$
\mathbb{E}_X[\varphi ' ( Q_{d, k_*}(a, X))] \le
\| \varphi ' \|_{\infty} \Bigl(C\bigl([a_d]^{-\frac{1}{dk_*}}+ 1\bigr) s^{\frac{1}{1+dk_*}} + 24d \exp \Bigl(-\frac{c[a_d]^2}{\delta(a_d)} \Bigr)\Bigr) +
\frac{2}{s} \| \varphi \|_\infty .
$$
The corollary is proved.
\end{proof}

\begin{corollary}
Under the assumptions of the previous corollary, one has the estimate
$$
|\mathbb{E}[\exp (i t Q_{d, k_*}(a, X))]| \le
C_0 \bigl([a_d]^{-\frac{1}{dk_*}}+ 1\bigr) |t|^{-\frac{1}{2 + dk_*}} + 48d \exp \Bigl(-\frac{c[a_d]^2}{\delta(a_d)} \Bigr)
$$
where $C_0 = C_0 (d, N, k_* , \omega, r, R).$
\end{corollary}
\begin{proof}
For $\varphi (x) = \sin tx$ and $\varphi (x) = - \sin tx$, we get
$$
\bigl|t \mathbb{E}[\cos (t Q_{d, k_*}(a, X))]\bigr| \le
| t | \Bigl(C\bigl([a_d]^{-\frac{1}{dk_*}}+ 1\bigr) s^{\frac{1}{1+dk_*}} + 24d \exp \Bigl(-\frac{c[a_d]^2}{\delta(a_d)} \Bigr)\Bigr) +
\frac{2}{s} .
$$
Thus,
$$
\bigl|\mathbb{E}[\cos (t Q_{d, k_*}(a, X))]\bigr| \le
C\bigl([a_d]^{-\frac{1}{dk_*}}+ 1\bigr) s^{\frac{1}{1+dk_*}} + 24d \exp \Bigl(-\frac{c[a_d]^2}{\delta(a_d)} \Bigr) +
\frac{2}{|t|s} .
$$
Then, for $s = |t|^{-\frac{1+dk_*}{2+dk_*}}$, we obtain
$$
\mathbb{E}[\cos (t Q_{d, k_*}(a, X))] \le
\widetilde{C} \bigl([a_d]^{-\frac{1}{dk_*}}+ 1\bigr)|t|^{-\frac{1}{2 + dk_*}} + 24d \exp \Bigl(-\frac{c[a_d]^2}{\delta(a_d)} \Bigr) .
$$
Similarly, one considers
$\varphi (x) = \cos tx$,
$\varphi (x) = - \cos tx$
and we obtain the estimate
$$
|\mathbb{E}[\exp (i t Q_{d, k_*}(a, X))] |
\le 2\widetilde{C}\bigl([a_d]^{-\frac{1}{dk_*}}+ 1\bigr) |t|^{-\frac{1}{2 + dk_*}} + 48d \exp \Bigl(-\frac{c[a_d]^2}{\delta(a_d)} \Bigr).
$$
The corollary is proved.
\end{proof}

We now show that using Theorem \ref{t2.1} directly one can get a better bound.

\begin{theorem}
Let $d, N, k_*\in \mathbb{N},\omega>0, r>0, R>0$.
Then there are positive numbers $c:=c(d,N, \omega, r)$ and
$C_1:=C_1(d, N, k_*, \omega, r, R)$,
dependent only on listed parameters, such that, for every sequence
$\{X_n\}_{n=1}^\infty$ of
independent $N$-dimensional random vectors with condition
$\mathfrak{D}(\omega, r, R)$
and for every coefficient function collection
$a=(a_1, \ldots, a_d)$, $a_m\colon (\mathbb{N}^3)^m\to \mathbb{R}$,
there is the estimate
$$|\mathbb{E}[ \exp (it Q_{d, k_*}(a, X))]|
\le
2C_1|t|^{-\frac{1}{dk_*}}[a_d]^{-\frac{1}{dk_*}}
+
24d \exp \Bigl(-\frac{c[a_d]^2}{\delta(a_d)} \Bigr)
$$
where $Q_{d, k_*}(a, X)$ is of the form \eqref{form}.
\end{theorem}
\begin{proof}
We note that
$$|\mathbb{E}[ \exp (it Q_{N, k_*}(a, X))]|
\le
|\mathbb{E}[\cos(t Q_{d, k_*}(a, X))]| + |\mathbb{E}[\sin(t Q_{d, k_*}(a, X))]|$$
and we now consider each summand separately.

We again make use of the representation from
Theorem \ref{t2.3}, and without loss of generality,
we assume that
$$
X_n \stackrel{Law}{=} \varepsilon_n V_n + (1 - \varepsilon_n)U_n
$$
where $\varepsilon_n$, $V_n$, $U_n$ are mutually independent,
$\varepsilon_n$ are Bernoulli random variables
with the probability of success $p = p(\omega, r, N)$,
$V_n$ are logarithmically concave random vectors,
and $U_n$ are some other random vectors.
By Theorem \ref{t2.1}, applied with the function $\varphi(x) = \sin (tx)$,
we get that
$$
\bigl|\mathbb{E}_V[t \cos (t Q_{d, k_*}(a, X) )]\bigr|
\le \frac{C(d, k_*)}{(\mathbb{D}_V[ Q_{d, k_*}(a, X)])^{\frac{1}{2dk_*}}} | t |^{1 - \frac{1}{dk_*}}.
$$
By Lemma \ref{L-1}, for $\theta = C[a_d]^2$,
$C:=C(d, N, k_*, \omega, r, R)$,
we get that
\begin{multline*}
|\mathbb{E}[\cos(t Q_{d, k_*}(a, X))]|
=
|\mathbb{E}_{\varepsilon, U}[\mathbb{E}_V[(I_{\mathbb{D}_V[Q_{d, k_*}(a, X)] \ge \theta} + I_{\mathbb{D}_V[Q_{d, k_*}(a, X)] < \theta})\cos(t Q_{d, k_*}(a, X))]]| \le \\
\le
C(d, k_*)|\mathbb{E}_{\varepsilon, U}[\theta^{-\frac{1}{2dk_*}}t^{-\frac{1}{dk_*}}]|
+ 2P_{\varepsilon, U}({\mathbb{D}_V[Q_{d, k_*}(a, X)] < \theta})
\\
\le
C_1|t|^{-\frac{1}{dk_*}}[a_d]^{-\frac{1}{dk_*}}
+
12d \exp \Bigl(-\frac{c[a_d]^2}{\delta(a_d)}\Bigr)
\end{multline*}
where $C_1:=C_1(d, N, k_*, \omega, r, R)$.
Similarly,
$$
|\mathbb{E}[\sin(t Q_{d, k_*}(a, X))]|
\le
C_1|t|^{-\frac{1}{dk_*}}[a_d]^{-\frac{1}{dk_*}}
+
12d \exp \Bigl(-\frac{c[a_d]^2}{\delta(a_d)}\Bigr).
$$
Therefore, we get
$$
|\mathbb{E}[ \exp (it Q_{d, k_*}(a, X))]|
\le
2C_1|t|^{-\frac{1}{dk_*}}[a_d]^{-\frac{1}{dk_*}}
+
24d \exp \Bigl(-\frac{c[a_d]^2}{\delta(a_d)}\Bigr).
$$
The theorem is proved.
\end{proof}

\section*{Acknowledgments}
This research was supported by the Russian Science Foundation Grant 22-11-00015
at Lo\-mo\-no\-sov Moscow State University.

\section*{Competing interests}
The authors have no competing interests to declare that are relevant to the content of this article.

\end{document}